\documentclass[11pt, reqno]{article}
\usepackage{amsfonts}
\usepackage{amsmath}
\usepackage{amsthm}
\usepackage{amssymb}
\usepackage{fancyhdr}
\usepackage{fancybox}
\usepackage{mathrsfs}

\def\phi{\varphi }
\def\epsilon{\varepsilon}

\swapnumbers

\theoremstyle{plain}
\newtheorem{theorem}{Theorem}[section]

\newtheorem{lemma}[theorem]{Lemma}

\theoremstyle{definition}
\newtheorem{definition}[theorem]{Definition}
\newtheorem{remark}[theorem]{Remark}
\newtheorem{remarks}[theorem]{Remarks}

\newtheorem{example}[theorem]{Example}

\numberwithin{equation}{section}

\begin{document}

\title{ A multivariate version of the disk convolution}
\author{Margit R\"osler\\Institut f\"ur Mathematik, Universit\"at Paderborn\\
Warburger Strasse 100,
D-33098 Paderborn\\
roesler@math.upb.de\\
and\\
Michael Voit\\
Fakult\"at Mathematik, Technische Universit\"at Dortmund\\
          Vogelpothsweg 87,
          D-44221 Dortmund, Germany\\
e-mail:  michael.voit@math.tu-dortmund.de}
\date{}

\maketitle

\begin{abstract}
We present an explicit product formula for the spherical
functions of the compact Gelfand pairs $(G,K_1)= (SU(p+q), SU(p)\times SU(q))$ with
$p\ge 2q$, which can be  considered as the elementary spherical functions of one-dimensional
$K$-type for the Hermitian symmetric spaces
$G/K$ with $K= S(U(p)\times U(q))$. Due to results of Heckman, they
can be expressed in terms of Heckman-Opdam
 Jacobi polynomials of type $BC_q$ with specific half-integer multiplicities.
 By analytic continuation with respect to the multiplicity parameters we  obtain
 positive product formulas for the extensions of these spherical functions as well as 
associated compact and commutative
hypergroup structures  parametrized by real $p\in]2q-1,\infty[$. We also obtain explicit
product formulas
    for the involved continuous two-parameter family of  Heckman-Opdam Jacobi polynomials with regular, but not necessarily positive
    multiplicities.
The results of this paper extend well known results for the disk convolutions for $q=1$ to higher rank.

\end{abstract}



%%%%%%%%%%%%%%%%%%%%%%%%%%

\section{Introduction}
It is well-known that
the spherical functions of Riemannian symmetric spaces of compact type
can be considered as Heckman-Opdam polynomials, which are the polynomial
variants of 
Heckman-Opdam hypergeometric functions.
In
particular, the spherical functions of Grassmann manifolds $SU(p+q,
\mathbb F)/S(U(p, \mathbb F)\times U(q,\mathbb F))$ with  $p\geq q$ and
$\mathbb F=\mathbb R,
\mathbb C, \mathbb H$  can be realized
as Heckman-Opdam polynomials of type $BC_q$
with certain mulitplicities given by the root data, namely
\[ k= (d(p-q)/2, (d-1)/2, d/2) \quad \text{with } \,d= \dim_\mathbb R \mathbb F \in\{1,2,4\},\]
corresponding to the short, long and middle roots, respectively. We
refer to  \cite{HS},
\cite{H} and
\cite{O1} for the foundations of Heckman-Opdam theory, and to 
\cite{RR} for the compact Grassmann case.  Recall that the spherical functions of a Gelfand pair $(G,K)$ can be characterized 
as the continuous, $K$-biinvariant functions on $G$ satisfying the product formula 
\begin{equation}\label{prodformel} \varphi(g)\varphi(h) = \int_K \varphi(gkh)dk\quad (g,h\in G)\end{equation}
where $dk$ is the normalized Haar measure on $K.$ 
In \cite{RR}, the
product formula for the spherical functions of  $SU(p+q,
\mathbb F)/S(U(p, \mathbb F)\times U(q,\mathbb F)),$ considered as functions on the
fundamental alcove 
\begin{equation}\label{A_q} A_q = \{t\in \mathbb R^q:\> \pi/2\ge t_1\ge
t_2\ge\ldots\ge t_q\ge0\},\end{equation}
was extended by
analytic interpolation to all real parameters  $p\in]2q-1,\infty[.$
This led to positive product formulas for the associated
$BC$-Heckman-Opdam polynomials as well as commutative hypergroups structures
on $A_q$ with the Heckman-Opdam polynomials as
characters. For a background on hypergroups see \cite{J}, where
hypergroups are called convos.
For the non-compact Grassmannians over $\mathbb R, \, \mathbb C$ and $\mathbb
H$, similar constructions had previously been
carried out in \cite{R2}. In the case $\mathbb F= \mathbb C$, where
the non-compact Grassmannian $SU(p,q)/ S(U(p)\times U(q))$ is a Hermitian
symmetric space, the approach of \cite{R2} was extended to
the spherical functions of  $U(p,q)/U(p)\times SU(q)$ in \cite{V9}. Here the
analysis was based on the decisive fact that the spherical
functions of this (non-symmetric) space are intimately related to the generalized spherical functions of
the symmetric space $SU(p,q)/ S(U(p)\times U(q))$ of
one-dimensional $K$-type. These in turn can be expressed in terms of
Heckman-Opdam hypergeometric functions according to  \cite{HS}, Section 5.

%For the general context of Heckman-Opdam theory, we
%refer to \cite{HS}, \cite{O1} and the
% references cited there. 

In the present paper, we consider the  pair
$(G,K_1)$ with $G=SU(p+q)$ and $K_1= SU(p)\times SU(q)$ (over
$\mathbb F=\mathbb C$), where we assume $p>q$. In this case, $(G,K_1)$ is a Gelfand pair.
By a decomposition of Cartan type,  we  identify the 
spherical
 functions of $G/K_1$ as functions on the compact cone
$$X_q:=\{(zr_1,r_2, \ldots, r_q): \quad 0\le r_1\le\ldots r_q\le 1, \>
z\in\mathbb T\}\subset \mathbb C\times\mathbb R^{q-1}$$
 with the
torus $\mathbb T:=\{z\in\mathbb C:\> |z|=1\}$. Similar as in the non-compact case \cite{V9}, 
the
spherical functions of $G/K_1$ can be considered as elementary spherical functions of $K$-type $\chi_l\, ( l\in \mathbb Z$) 
for the Hermitian symmetric space
$G/K$ with $K= S(U(p)\times U(q)).$ According to the results of Heckman \cite{HS}, 
 they  can  be expressed in terms of
 of Heckman-Opdam polynomials of type $BC_q$, depending on the
integral  parameter $p>q$ and the spectral parameter $l\in\mathbb Z$.
Under the stronger requirement $p\geq 2q$, we proceed similar as in \cite{RR} and write down the
 product formula for the spherical functions of $G/K_1$
 as  product formulas on $X_q$.
 
In Section 4 we then extend this formula to a product formula 
on $X_q$ for a continuous range of parameters $p\in] 2q-1,\infty[$    by means of
Carlson's theorem, a  principle of analytic continuation.
In particular, we obtain a continuous family of associated
commutative hypergroup structures on $X_q$.
We  determine the dual spaces and Haar measures of these hypergroups.
  For $q=1$, the space  $X_q$ is the complex unit disk,
 and the associated hypergroups are the
well-known disk hypergroups studied in \cite{AT}, \cite{Ka},
 and  \cite{BH}, where the associated product formulas are based on the work of
 Koornwinder \cite{K2}.

For  each real $p\in] 2q-1,\infty[$, the associated hypergroup structure on $X_q$ 
contains a
 compact subgroup isomorphic to the one-dimensional torus $(\mathbb T,\cdot)$. 
Moreover, the quotient
$X_q/\mathbb T$ can be identified with the alcove $A_q$ and carries 
associated  quotient convolution structures; 
see \cite{J} and \cite{V1}  for the general
background. In this way  we  in particular 
recover the above-mentioned hypergroup structures of  \cite{RR} on
$A_q$ for $\mathbb F=\mathbb C$.
More generally, 
we obtain in Section 5 explicit continuous 
product formulas and convolution structures on $A_q$ for
all Jacobi polynomials of type $BC_q$ with multiplicities
$$k=(k_1,k_2,k_3)= (p-q-l,1/2+l, 1)$$
with  $p\in] 2q-1,\infty]$ and $l\in\mathbb R$.
Unfortunately,  for general $l\ne0$, the positivity of these product formulas
remains
open. 

\medskip\noindent
\textbf{Acknowledgement.} It is a pleasure to thank Maarten van Pruijssen for valuable hints 
concerning Gelfand pairs. 

\section{Preliminaries}

We start our considerations with the compact Grassmann manifolds
$G/K$ over $\mathbb C$, where $G= SU(p+q)$ and
$K=S(U(p)\times U(q))$ with $p\geq  q\geq 1.$ From the Cartan
decomposition of $G $ (see \cite{RR} or Theorem VII.8.6 of
\cite{Hel}) it follows that a system of representatives of the
double coset space $G //K$ is given by 
the matrices
\[ a_t =
\begin{pmatrix} I_{p-q}&0&0\\0 & \cos \underline t & -\sin \underline t
 \\0 &\sin \underline t & \cos \underline t \end{pmatrix},\quad t\in A_q\]
with $A_q$ as in \eqref{A_q};
here
$\underline t$ is the $q\times q$-diagonal matrix with the components of $t$ as
entries, and
$\cos\underline t,\, \sin\underline t$ are understood componentwise. We recall
from \cite{RR} how the double coset representatives can be determined
explicitly:

For $X\in M_q(\mathbb C)$ denote by $\,\sigma_{sing}(X) =
\sqrt{\text{spec}(X^*\!X)} = (\sigma_1 \ldots, \sigma_q) \in \mathbb R^q$ the vector of
singular values of $X$, decreasingly ordered by size.
Write $g\in G$ in $(p\times q)$-block notation as
\begin{equation}\label{block} g = \begin{pmatrix} A(g) & B(g)\\
        C(g) & D(g)
       \end{pmatrix},\end{equation}
and suppose that $g \in K b_tK.$ Then
\begin{equation}\label{xspec} t =
\text{arccos}(\sigma_{sing}(D(g))).\end{equation}

 By \cite{RR}, the spherical functions of $G/K$  are 
given by Heckman-Opdam polynomials, as follows: Consider $\mathbb R^q$ 
with the standard
inner product $\langle \,.\,,\,.\,\rangle$ and denote by $F_{BC_q}(\lambda,k;t)$ 
the Heckman-Opdam hypergeometric function
associated with the root system 
\[R=2BC_q = 
\{ \pm 2e_i, \pm 4e_i, 1\leq i\leq q\}\cup\{ \pm 2e_i \pm 2e_j: 1\leq i<j\leq q\}\subset\mathbb R^q\] 
and with multiplicity parameter $k=(k_\alpha)_{\alpha \in R}$, in the notion of \cite{R2}. 
Recall that there exists an open regular set $K^{reg}\subset \mathbb C^3$  
such that 
$F_{BC_q}(\lambda,k;\,.\,)$ exists in a suitable tubular neighbourhood of 
$\mathbb R^q\subset \mathbb C^q.$ We also write
$k=(k_1, k_2, k_3)$ where $k_1, k_2, k_3$ belong to the short, long and middle roots, respectively. 
Fix the positive subsystem 
$\,R_+ = \{ 2e_i, 4e_i, 2e_i \pm 2e_j: 1\leq i<j\leq q\}.$
Writing
$\alpha^\vee:=2\alpha/\langle\alpha,\alpha\rangle$, 
the associated set of
dominant weights is   
\begin{equation}\label{dominantweights} P_+ := \{ \lambda \in \mathbb R^q: 
\langle\lambda,\alpha^\vee\rangle \in \mathbb Z_+
\,\,\forall\,\alpha\in R_+\}=\{\lambda\in 2\mathbb Z_+^q:\> \lambda_1\ge
\lambda_2\ge\ldots\ge\lambda_q\}.\end{equation}
Notice that our normalization of root sytems and multiplicities is in 
accordance with \cite{HS}, \cite{O1} but differs from 
the ``geometric'' notion of \cite{RR}, where both are rescaled by a factor $2$. 

We consider the renormalized Heckman-Opdam polynomials associated with $R$ and $k$ in 
trigonometric notion as in \cite{RR},  which are defined by
\begin{equation}\label{HOrenorm} R_\lambda(k;t) = F_{BC_q}(\lambda + \rho(k),k;it), \quad 
\lambda \in P_+ \,.\end{equation}
Here
\[ \rho(k) = \frac{1}{2}\sum_{\alpha \in R_+} k_\alpha \alpha\]
and
\[
c(\lambda,k) = \prod_{\alpha\in R_+}
\frac{\Gamma\bigl(\langle\lambda,\alpha^\vee\rangle +
\frac{1}{2}k_{\alpha/2}\bigr)}{\Gamma \bigl(\langle\lambda,\alpha^\vee\rangle +
\frac{1}{2}k_{\alpha/2} + k_\alpha\bigr)}\cdot \prod_{\alpha\in R_+} 
\frac{\Gamma
\bigl(\langle\rho(k),\alpha^\vee\rangle + \frac{1}{2}k_{\alpha/2} +
k_\alpha\bigr)}{\Gamma\bigl(\langle\rho(k),\alpha^\vee\rangle +
\frac{1}{2}k_{\alpha/2}\bigr)}\,\]
is the generalized $c$-function, with the convention 
 $\,k_{\alpha/2} := 0$ if $\,\alpha/2\notin R$. 
 The Heckman-Opdam polynomials are holomorphic on $\mathbb C^q$. Indeed, $F_{BC_q}(\lambda + \rho(k),k;\,.\,)$ is 
 holomorphic on all of $\mathbb C^q$ if and only if $\lambda \in P_+$, see \cite{HS}.
 
According to Theorem 4.3. of \cite{RR}, the spherical functions of $G/K=SU(p+q)/S(U(p)\times U(q))$ 
can be considered as trigonometric polynomials on the alcove $A_q\subset \mathbb R^q$ and are  given by
\[ \varphi_\lambda (t) = R_\lambda(k;t), \quad \lambda \in P_+\]
with the multiplicity \[k = (k_1, k_2, k_3) = (p-q, 1/2, 1).\]
For  fixed $q\ge1$ and $p\geq 2q$,   the  product formula
 (\ref{prodformel}) for the $\varphi_\lambda$ was written in  \cite{RR}
 as a product formula 
  on $A_q$ depending on  $p$
 in a way which allowed extension to all 
 real parameters $p> 2q-1$ by  analytic continuation; see Theorem 4.4. of \cite{RR}.

\section{Spherical functions of  $(SU(p+q),
SU(p)\times SU(q))$ and their product formula}

Let us turn to the pair $(G, K_1):=(SU(p+q), SU(p)\times SU(q)), $ where we assume $p>q$. In this case,
$(G,K_1)$ is a Gelfand pair according to the classification of \cite{Kr}. 
In this section, we derive an explicit product formula for the spherical functions of $(G, K_1)$.
First, we determine a
system of 
representatives for the double coset space $G//K_1$.
%For $z\in \mathbb T$ and $r = (r_1, \ldots, r_q) \in \mathbb R^q$ we write
%\[ \underline r(z) = \text{diag}(zr_1, r_2, \ldots, r_q)\in M_q(\mathbb C)\]
%and introduce the compact cone
%\[ C_q := \{ \underline r(z): z\in \mathbb T, \, r\in \mathbb R^q \,\text{ with
%}\, 0 \leq r_1 \leq \ldots \leq r_q\leq 1\}\subseteq \mathbb C\times \mathbb
%R^{q-1}.\]
For this, consider the compact set
\[ X_q := \{x= (zr_1,r_2,\ldots, r_q): z\in \mathbb T, \,r_i\in \mathbb R \text{
with }0\leq r_1 \leq \ldots \leq r_q \leq 1\}\subseteq \mathbb C\times
\mathbb R^{q-1}.\]
If $q=1$, this is just the closed unit disc in $\mathbb C$. For
$q\geq 2$, the set $X_q$ can be interpreted as  a cone of real dimension $q+1$
with $X_{q-1}$ as basis set.

For $x=(zr_1, r_2, \ldots, r_q) \in X_q$ we write 
\[x=[r,z] \,\,\text{ with } \,r= (r_1, \ldots, r_q) .\]  
Note that the phase factor $z\in
\mathbb T$ is  arbitrary if $r_1 = 0.$
Now consider the alcove $A_q$. We define an equivalence relation on $A_q\times
\mathbb T$ via
\[ (t,z) \sim (t^\prime, z^\prime ) :\Longleftrightarrow\, t= t^\prime \,\text{
and }\, z = z^\prime \text{ if }\, t_1 = t_1^\prime <\frac{\pi}{2}.\]
Then the mapping $\, A_q\times \mathbb T \,\to X_q\,, \, (t,z) \mapsto [\cos t, z]$ induces a homeomorphism 
between the quotient space $(A_q\times \mathbb T)/\sim\,$ and the cone $X_q$.

%In the following, we shall denote the determinant of a complex square matrix $A$ by $\Delta(A).$ 

%For $z\in \mathbb T$, put  $h(z) = \text{diag}(z, 1\ldots, 1) \in M_q(\mathbb
%C)$ and let $H_q := \{ h(z): z\in \mathbb T\}\subset U(q).$ Then $U(q) =
%SU(q)\rtimes H_q$.

\begin{lemma} 
 For $z\in \mathbb T$, let $h(z) = \text{diag}(z, 1\ldots, 1) \in
M_q(\mathbb C).$ Then a set of
representatives of $G// K_1$ is given by the matrices $b_x, x\in X_q$
with 
\begin{equation}\label{b-x-z-def}
b_{x} =  \begin{pmatrix} I_{p-q}&0&0\\0 &  h(z^{-1})\,\cos\underline  t
\,&
-\sin\underline t
 \\0 & \sin\underline t &  h(z)\cos \underline t\,\end{pmatrix} \text { for }\,x
= [\cos t, z] \,\text{ with } t\in A_q.
\end{equation}
In particular, the double coset space $G//K_1$ is naturally homeomorphic with
the cone $X_q$.
\end{lemma}

\begin{proof}
We first check that each double coset has a representative of the stated
form.
In fact, by the known Cartan decomposition of $G$ with respect to $K$,   each $g\in G$
can be  written as
\[g=\begin{pmatrix} u_1&0\\0&v_1\end{pmatrix}  a_t
\begin{pmatrix} u_2&0\\0&v_2\end{pmatrix}\]
with $t\in A_q$ and 
$u_i\in U(p)$, $v_i\in U(q)$ satisfying $\det(u_i)\cdot \det(v_i) = 1.$
As $U(q) = SU(q) \rtimes H_q$ with $H_q = \{ h(z): z\in \mathbb T\}$, there are
$z_1, z_2 \in \mathbb T$ such
that $ \widetilde v_1 := v_1h(z_1)^{-1}\in SU(q)$ and $\, \widetilde v_2 :=
h(z_2)^{-1}v_2\in SU(q).$
For $z,w\in \mathbb T,$ define
 $\, H(w,z)= \text{diag}(w,1, \ldots, 1, z, 1, \ldots, 1)
\in M_{p}(\mathbb C)$ with the entry $z$ in position $p+1$.
Put $\,\widetilde u_1 := u_1\cdot H(z_1/z_2,z_2)$ and $\,\widetilde
u_2 := H(z_2/z_1, z_1)\cdot u_2\in U(p).$ Then 
$\widetilde u_1, \widetilde u_2 \in SU(p)$, and 
a short calculation in $p\times q$-blocks
gives
\begin{align*} g \,=\,& \begin{pmatrix}
u_1 & 0 \\
0 & v_1 \end{pmatrix}
\begin{pmatrix}
 I_{p-q}&0&0\\0 & \cos \underline t & -\sin \underline t
 \\0 &\sin \underline t & \cos \underline t \,
\end{pmatrix}
\begin{pmatrix} u_2 & 0 \\
              0 & v_2 
\end{pmatrix} \\
\,=\, & \begin{pmatrix}
\widetilde u_1 & 0 \\
0 & \widetilde v_1 \end{pmatrix}
\begin{pmatrix}
 I_{p-q}&0&0\\0 & \,h(z_1)^{-1}h(z_2)^{-1} \cos \underline t\, & \,-\sin
\underline t
 \\0 &\,\sin \underline t \,& \, h(z_1)h(z_2)\cos \underline t \,
\end{pmatrix}
\begin{pmatrix} \widetilde u_2 & 0 \\
              0 & \widetilde v_2
\end{pmatrix}.
\end{align*}
Thus $g\in K_1 b_x K_1$ with $x= [\cos t, z_1z_2]$.
%\begin{align*} g &= \begin{pmatrix} u_1 \begin{pmatrix} I_{p-q} & 0 \\
 %                                         0 & \cos \underline x\,
 %                          \end{pmatrix}u_2    \,&\, u_1 \begin{pmatrix} 0\\
  %                                       -\sin\underline x\,\end{pmatrix} v_2\\
%v_1 \begin{pmatrix} 0 & \sin \underline x\,\end{pmatrix} u_2 \, &\,  v_1
%\cos\underline x\, v_2\,
%\end{pmatrix}\, \\
%& = \begin{pmatrix} \widetilde u_1 \begin{pmatrix} I_{p-q} & 0 \\
%                                          0 & h_1h_2\cos \underline x\,
%                           \end{pmatrix}\widetilde u_2    \,&\, \widetilde u_1
%\begin{pmatrix} 0\\
%                                         -\sin\underline x\,\end{pmatrix} v_2\\
%\widetilde v_1 \begin{pmatrix} 0 & \sin \underline x\,\end{pmatrix}
%\widetilde u_2 \, &\, \widetilde v_1
%(h_1h_2)^{-1}\cos\underline x\, \widetilde v_2\,
%\end{pmatrix}
%\end{align*}

In order to show  that the $b_{x}$ are contained in different double cosets
 for different $x,$ we analyze how the parameter $x$ depends on $g\in G$. Let
\[g = \begin{pmatrix} \widetilde u_1 & 0 \\
       0 & \widetilde v_1 
      \end{pmatrix} b_{x}
\begin{pmatrix} \widetilde u_2 & 0 \\
       0 & \widetilde v_2
      \end{pmatrix}  \]
with  $\widetilde u_i \in SU(p), \,\widetilde v_i \in SU(q).$ Suppose that $x= [r,z]$ with $r= \cos t, \, t\in A_q$. Using the
$(p\times q)$-block notation \eqref{block}, we obtain
\[ D(g)  = \widetilde v_1 h(z) \underline r \,\widetilde v_2\,, \quad
 \Delta(D(g)) = z\cdot\prod_{i=1}^q r_i\,= \, z\cdot |\Delta(D(g))|,\]
 Here and throughout the paper, $\Delta$ denotes the usual determinant of a complex square matrix.
 Thus $\, r= \cos t\, = \sigma_{sing}(D(g)).$ Further, $z$ is uniquely determined by $g$ exactly if $r_1\not=0$,
 which is equivalent to $\prod_{i=1}^q r_i \not = 0. $ In this case, $ \, z= \frac{\Delta D(g)}{|\Delta D(g)|}.$
Thus $(t,z)\in A_q\times \mathbb T$ is uniquely determined  by $g$ up
to the equivalence $\sim$, which proves our statement.

\end{proof}

%We have not yet clarified whether $(U, \widetilde K)$ is a Gelfand pair. In
%order to
%check this, we prove that the (usual) convolution of $\widetilde K$-biinvariant
% measures on $U$ is commutative.
%For this, it suffices to consider convolution products of uniform distributions
%on double cosets, which are of the form
%\[ \mu_x = \_{\widetilde K}\int_{\widetilde K} \delta_{kb_xl} dkdl, \quad
%x\in X_q\,.\]
%(Both integrals are with respect to the normalized Haar measure of
%$\widetilde K$.) For any $\widetilde K$-invariant $f\in C(U)$, one easily
%checks that
%\begin{equation}\label{Flatung_masse}(\mu_x*\mu_y)(f) = \int_{\widetilde K}
%f(b_xkb_y)dk.\end{equation}
%In the following, we shall write this integral in an explicit form,
%from which shall infer that $\, \mu_x*\mu_y = \mu_y*\mu_x$.

%\begin{theorem}
% Suppose $f\in C(U)$ is $\widetilde K$-biinvariant. Then
%\[\int_{\widetilde K} f(b_xkb_y)dk \,=\, .....\]

%\end{theorem}

%*************
%\[M^1(U\vert\vert\widetilde K) =\{ \mu\in M^1(U): \delta_k *\mu*\delta_l = \mu
%\quad\text{for all } \, k,l\in \widetilde K\}\]
% denote the set of $\widetilde
%K$-biinvariant measures from $M^1(U)$, which forms a subalgebra of $(M^1(U),
%\ast).$ Let $\mu, \, \nu \in M^1(U\vert \vert \widetilde K)$. Then for
%$\widetilde K$-biinvariant 
%$f\in C(U),$ one has
%******************

The proof of the above lemma reveals  the following equivalence for $g\in G$: 
\[ g\in K_1b_x K_1 \, \Longleftrightarrow \, x = [r,z]\,\text{ with }\, r= \sigma_{sing}(D(g)), \, 
z = \arg\Delta(D(g)),\]
with the argument $\,\arg:\mathbb C \to \mathbb T$  defined by 
\[ \arg\,z := \frac{z}{|z|}\, \text{ if } z\not = 0, \,\, \arg\,0\,:= 1.\]

  We are now going to write the general product formula (\ref{prodformel}) 
for spherical functions as a product formula on the cone $\,X_q\cong (A_q\times \mathbb T)/\sim\,.$ 
For $x = [\cos  t, z_1], \, x^\prime = [ \cos t^\prime, z^\prime]\in X_q$ and 
$K_1$-biinvariant $f\in C(G)$ we have to evaluate the integral 
\[\int_{ K_1} f(b_{x}kb_{x^\prime})\,dk.\]
Write 
$\,\displaystyle k= \begin{pmatrix} u & 0\\
0 & v     \end{pmatrix}\, $ with $u\in SU(p), \, v\in SU(q).$
Then $(p\times q)$-block calculation gives 
\[b_{x}kb_{x^\prime} =\begin{pmatrix} * & *\\
* & D(b_{x}kb_{x^\prime}) \end{pmatrix}\]
where
\[
 D(b_{x}kb_{x^\prime}) = 
(0,\sin\underline t\,)\,u \begin{pmatrix}0\\ -\sin\underline t^\prime\end{pmatrix}
+h(z) \cos \underline t \,v\, h(z^\prime)\cos \underline t^\prime\,  .\]
With the block matrix 
%\begin{equation}\label{def-sigma-0}
 \[\sigma_0 := \begin{pmatrix}0_{(p-q)\times q}\\ I_q\end{pmatrix} \in M_{p,q}(\mathbb C)\]
%\end{equation}
this can be written as 
\begin{equation}\label{D(.)}
 D(b_{x} k b_{x^\prime}) = \,-\sin\underline t \,\sigma_0^* u
\sigma_0\sin\underline t^\prime \,+\,
h(z z^\prime)\,\cos \underline t\, v\, \cos \underline t^\prime .\end{equation}
Regarding
 $K_1$-biinvariant  functions  $f\in C(G)$ as  continuous
functions  on $X_q$, we have
\[
\int_{K_1} f(b_{x}kb_{x^\prime})dk\,=\,
\int_{ SU(p)\times SU(q)} f\left([\sigma_{sing}(D(b_xkb_{x^\prime})),
\arg\Delta(D(b_{x}k b_{x^\prime})\,]\right)dk
\]
with $D(b_{x}kb_{x^\prime})$ from \eqref{D(.)}.
 Notice that 
$\,\sigma_0^* u \sigma_0\in M_q(\mathbb C)$ is 
 the lower right $q\times q$-block of $u$ and is contained in the closure of the ball
$$B_q:=\{w\in M_q(\mathbb C):\> w^*w\le I_q\},$$
where $w^*w\le I_q$ means that $I_q-w^*w$ is positive semidefinite. 
We now assume that $p\geq 2q$ and reduce the $SU(p)$-integration by means of Lemma 2.1 of \cite{R2}. Notice first that for
continuous $g$ on $\overline B_q$,
 \[\int _{SU(p)} g(\sigma_0^*u\sigma_0) du = \int _{U(p)} g(\sigma_0^*u\sigma_0) du,\]
where $du$ denotes the normalized Haar measure in each case.  Thus Lemma 2.1 of \cite{R2} gives
 
\begin{align}
\int_{K_1} f(b_{x}kb_{x^\prime})dk&\,=\,\frac{1}{\kappa_p}\int_{B_q}\int_{SU(q)} 
f\bigr(\big[\sigma_{sing}(-\sin\underline t \,w\sin\underline t^\prime +
h(zz^\prime)\cos \underline t\, v\cos \underline t^\prime
), \notag \\
&\arg\Delta(-\sin\underline t \,w \sin\underline t^\prime +
h(zz^\prime)\cos \underline t\, v \cos \underline t^\prime
)\big]\bigr)
\Delta(I_q-w^*w)^{p-2q} dv dw,
\end{align}
 where
\[
\kappa_p=\int_{B_q}\Delta(I_q-w^*w)^{p-2q}\> dw
\]
 and 
 $dw$ means integration with respect to Lebesgue measure on $B_q$. After the substitution
$w\mapsto
 h(zz^\prime)w$, we finally arrive at the following

%\begin{proposition}\label{prop-group-torus-convo}
\begin{theorem}\label{prop-group-torus-convo-Xq}
Suppose that $p\ge 2q$. Then the product formula for the spherical functions of the Gelfand pair 
$(G,  K_1)$, considered 
as functions on the cone $X_q$, can be written as
\begin{align}\label{group-torus-convo}
\phi([\cos t,z])\phi([\cos t^\prime, z^\prime]) \,=\, 
\frac{1}{\kappa_p}&\int_{B_q}\int_{SU(q)} 
\phi\bigr(\big[\sigma_{sing}(-\sin\underline t \,w\sin\underline t^\prime +
\,\cos \underline t\,v\cos \underline t^\prime),\notag  \\
zz^\prime\cdot \arg\Delta(&-\sin\underline t \,w\sin\underline t^\prime +
\cos \underline t\, v \cos \underline t^\prime)\big]\bigr)\cdot
\Delta(I_q-w^*w)^{p-2q}\, dvdw.\notag
\end{align}
\end{theorem}

We remark that 
for $p=2q-1$, a degenerate version of this integral formula may be written down by using the coordinates introduced  
in Section 3 of \cite{R1}.

%\begin{theorem}\label{prop-group-torus-convo-Xq}
%Let $p\ge 2q$.
%%If a $\tilde K$-spherical function $\phi\in C(G)$ is regarded as a continuous
%function on $X_q$, then the associated product
%formula for spherical functions $\phi$ has the following form  for
% $(r,z_1):=(z_1r_1,r_2,\ldots,r_q)$ , $ (s,z_2):=(z_2s_1,s_2,\ldots,s_q)\in
%X_q$:
%\begin{align}\label{group-torus-convo-Xq}
%\phi(r,z_1)\phi(s,z_2)=\quad\quad &\\
%=\frac{1}{\kappa_p}\int_{B_q}\int_{SU(q)} 
%\phi\Bigr(\sigma_{sing}(&  \underline r\> v\> \underline s-\sqrt{1-\underline
%  r^2} \,w\,  \sqrt{1-\underline s^2} ),\notag  \\
%z_1z_2\cdot \arg(\det(&\underline r\, v \underline s-\sqrt{1-\underline
%  r^2} \,w\,  \sqrt{1-\underline s^2} ))\Bigl)\cdot
%\det(I_q-w^*w)^{p-2q}
%\> dv\> dw\notag
%\end{align}
%with $\underline r:=diag(r_1,\ldots, r_q),\underline s:=diag(s_1,\ldots, s_q)$.
%\end{theorem}

\medskip

We next turn to the classification of the  spherical functions of  $(G, K_1)$. 
 Note first that $K = K_1\rtimes H$ with $H = \{ H_z, z\in \mathbb T\}$ where $H_z$ is 
 the diagonal matrix with entries $z$ in position $p-q+1$, $1/z$ in position $p+1$ and $1$ else.  
Let $\chi:K \to \mathbb T$ be the homomorphism
 with kernel $ K_1$ and $\chi(H_z):= z.$ 
 Then the characters of $K$ are given by the functions $\chi_l(k) = \chi(k)^l, \,l\in \mathbb Z$, and we have the following characterization:
\begin{lemma}\label{euiv-def-spher}
 For $\phi\in C(G)$ the following properties are equivalent: \parskip=-1pt
\begin{enumerate}\itemsep=-1pt
\item[\rm{(1)}] $\phi$ is $K_1$-spherical, i.e., $K_1$-biinvariant
with 
 $\phi(g)\phi(h)=\int_{K_1} \phi(gkh) dk\,$ for all $g,h\in G$.
\item[\rm{(2)}] $\phi$ is an elementary
 spherical function for $(G,K)$ of type $\chi_l$  for some $l\in\mathbb Z$, i.e. 
 $\phi$ is not identical zero and satisfies the twisted product formula
 \begin{equation}\label{twist-prod}
\phi(g)\phi(h)=\int_{ K} \phi(gkh)\chi_l(k)dk \quad\text{ for all }\,
g,h\in G.
\end{equation}
 
\end{enumerate}
\end{lemma}

\begin{remark}
We here adopt the  notion of elementary spherical functions of type $\chi_l$ 
according to \cite{HO}. Each such function automatically satisfies $\varphi (e) = 1$ as well as
the 
 $\chi_l$-bi-coinvariance condition
\begin{equation}\label{twist} \varphi(k_1gk_2) = \chi_l(k_1k_2)^{-1} \cdot \varphi(g) \quad \text {for all }\, g\in G, 
k_1, k_2 \in K,\end{equation}
 see Lemma 3.2. of \cite{HO}.
The discussion of \cite{HS} is based on a different, but equivalent definition of 
elementary spherical functions of $K$-type (for the non-compact dual), which requires  \eqref{twist} together 
with a  system of invariant differential operators on sections in an associated homogeneous line bundle,
see Definition 5.2.1 of \cite{HS}.  For the equivalence of definitions we refer to
Theorem 3.2 of \cite{S}.
\end{remark}

\begin{proof}[Proof of Lemma \ref{euiv-def-spher}.] The proof of this result, which we expect to be well-known,  can be carried out
 for instance precisely as in Lemma 2.3 of \cite{V9}. 
\end{proof}

For an arbitrary irreducible Hermitian symmetric space and its compact dual, 
the elementary spherical functions of type $\chi_l$ can be written as
modifications of Heckman-Opdam hypergeometric functions,
see Section 5  of \cite{HS}, in particular Theorem 5.2.2 and Corollary 5.2.3., as well as \cite{HO}.
In the compact case, they correspond to the $\chi_l$-spherical representations of $G$ which were classified in \cite{Sch}. 
To become explicit in the particular case of our compact symmetric spaces $G/K= SU(p+q)/S(U(p)\times U(q)),$ 
recall the set  $P_+$ of dominant weights from \eqref{dominantweights}
as well as the renormalized Heckman-Opdam polynomials $R_\lambda$ of type $BC_q$ from \eqref{HOrenorm}.
According to \cite{HS}, the elementary spherical functions of $(G,K)$ of type $\chi_l$,  
considered as functions on $A_q$, are given by 
\begin{equation}\label{chispher}t\mapsto  \prod_{j=1}^q \cos^{|l|}\! t_j\cdot
 R_\lambda(k(p,q,l);t),  \quad\lambda\in P_+\,,\, l\in\mathbb Z\end{equation}
with multiplicity parameters
\[k(p,q,l):=(p-q-|l|, 1/2+|l|, 1), \quad l\in \mathbb Z.\]
 Notice at this point that for $F_{BC_q}$, the set
\[ \{ k= (k_1,k_2,k_3): \text{Re}\, k_3 \geq 0, \,\text{Re}(k_1+k_2) \geq 0\}\]
is contained in $K^{reg}$ (see Remark 4.4.3 of \cite{HS}), and so in particular
the multiplicities $k(p,q,l)$ are regular. The associated $\rho$-function is
\begin{equation*}
\rho(k(p,q,l))=
(p-q +|l|+1)\sum_{j=1}^q e_j\,+ 2 \sum_{j=1}^q (q-j)e_j.
\end{equation*}
We now turn to the spherical functions of  $(G,K_1)$, which we again consider 
as functions on the cone $X_q$.

\begin{theorem}\label{classification-spher-allg} Let $(G,K) = (SU(p+q),S(U(p)\times U(q)))$ and $K_1 = SU(p)\times SU(q).$ 
Then the spherical functions of the Gelfand pair 
$(G, K_1)$, considered as functions on $X_q$, are precisely given
by
\begin{equation}\label{def-spherical-jacobi}
\phi_{\lambda,l}^p([\cos t,z])=z^l\cdot\prod_{j=1}^q \cos^{|l|}\! t_j\cdot
 R_\lambda(k(p,q,l);t),  \quad\lambda\in P_+\,,\, l\in\mathbb Z
\end{equation}
 with the multiplicity
$k(p,q,l)=(p-q-|l|,\> \frac12 +|l|,\> 1)\in K^{reg}.$
 
\end{theorem}

\begin{proof}

First observe that $\phi_{\lambda,l}^p$ is indeed well-defined as a function on $X_q$, 
because 
the right side is zero if $t_1 = \pi/2$, independently of $z\in \mathbb T.$
Now suppose that $\varphi: G\to \mathbb C$ is spherical for $(G,K_1)$. Then by Lemma \ref{euiv-def-spher} it is 
$\chi_l$-spherical for some $l\in \mathbb Z$. Consider $x=[\cos  t, z] \in X_q$ and write
$b_x = H_{1/\sqrt{z}}\,a_tH_{1/\sqrt{z}}$ with an arbitrary square root of $z$. Then in view of \eqref{twist}, 
\[ \varphi(b_x) = z^l \cdot \varphi(a_t),\]
where $\varphi(a_t)$ is of the form \eqref{chispher}. This proves the statement.
\end{proof}

%For the converse statement consider the $\mathbb C$-vector space $V\subset
%C((A_q\times\mathbb T)/\sim)$ spanned by the functions $\phi_{\lambda,l}^p $ 
%($\lambda\in P_+$, $l\in\mathbb Z$). Clearly, this space is invariant under
%complex conjugation and products. Moreover, by the definition of  the space
%$(A_q\times\mathbb T)/\sim$ it is clear that the  functions $\phi_{\lambda,l}^p
%$ 
%($\lambda\in P_+$, $l\in\mathbb Z$) separate points. Therefore, by the 
% Stone-Weierstrass theorem, $V$ is $\|.\|_\infty$-dense in
 %$C((A_q\times\mathbb T)/\sim)=C(G//\tilde K)$. The well known orthogonality
% of the spherical functions in $L^2(G//\tilde K)$ now ensures that there are
% no further spherical functions.

\begin{example}[The rank-one case $q=1$] \label{q1--example-1}
Here $G/K = SU(p+1)/S(U(p)\times U(1))\cong U(p+1)/U(p)$ and $G//K$ is homeomorphic to the unit disk 
$X_1=\{z\in \mathbb C: |z|\leq 1\}~=~D$. 
We shall identify the spherical
functions  $\phi_{\lambda,l}^p $  as the well-known disk polynomials on $D$ introduced in \cite{K2}, which are known to be
the spherical functions of $(U(p+1),U(p)).$ 

We have $R_+ = \{ 2,4\} \subset \mathbb R$ and $P_+=2\mathbb Z_+$. According to the example on p.89f of \cite{O1}, 
$F_{BC_1}(\lambda,
k;t)$ may be expressed as a  $_2F_1$- (Gaussian) hypergeometric function. 
 Consider the renormalized  one-dimensional Jacobi
polynomials
\begin{equation}\label{classical-jacobi-pol}
R_n^{(\alpha,\beta)}(x)= \>_2F_1(\alpha+\beta+n+1, -n, \alpha+1; (1-x)/2)
\quad(x\in\mathbb R, \> n\in\mathbb Z_+)
\end{equation}
for $\alpha,\beta>-1.$ Then the Heckman-Opdam polynomials associated with $R=2BC_1$ and multiplicity $k=(k_1,k_2)$ can be written as
\begin{equation}\label{ident-jacobi}
R_{2n}(k;t)=R_n^{(\alpha,\beta)}(\cos 2t)\quad\quad (n\in \mathbb Z_+, \>
t\in[0,\pi/2])
\end{equation}
with 
$$\alpha=k_1+k_2-1/2, \quad \beta= k_2-1/2;$$
c.f. equation (5.4) of \cite{RR}, where a different scaling of roots and multiplicites is used. 
Writing $r=\cos t $, we have $\cos 2t = 2r^2-1$. 
We thus obtain from Theorem \ref{classification-spher-allg} the
well-known fact that the spherical functions of $(G,K)$ are given, as functions on $D$,  by
the  so-called disk polynomials
\begin{equation}\label{disk-polynimials}
\widetilde\phi_{n,l}(zr)= \varphi_{2n, l}^p([r,z]) = z^l r^{|l|}\cdot R_n^{(p-1,|l|)}(2r^2-1) \quad\quad
(z\in\mathbb T,\> r\in[0,1])
\end{equation}
with $l\in\mathbb Z$, $n\in\mathbb Z_+$.
Moreover, the product formula of Theorem \ref{prop-group-torus-convo-Xq} becomes in this case  
\begin{equation}\label{disk-polynimial-product-formulas}
\widetilde\phi_{n,l}(zr)\cdot \widetilde\phi_{n,l}(z^\prime s)=
\frac{1}{\kappa_p}\int_{D} 
\widetilde\phi_{n,l}(zz^\prime (rs-w\sqrt{1- r^2}  \sqrt{1- s^2} ))\cdot
(1-|w|^2)^{p-2}\> dw
\end{equation}
with 
$$\kappa_p=\int_{D}(1-|w|^2)^{p-2}\> dw= \frac{\pi}{p-1}. $$ 
This formula is well known; see for instance \cite{AT}, \cite{Ka}.
\end{example}

\section{Convolution algebras on the cone $X_q$ for a continuous family of multiplicities}

In this section, we extend the product formula for the spherical
 functions of $(G,K_1)$  in Theorem
 \ref{prop-group-torus-convo-Xq} from integers $p\ge 2q$ to a continuous range of 
parameters
 $p\in]2q-1,\infty[$.
 We show that for each $p\in]2q-1,\infty[$,
the corresponding product formula induces a commutative Banach algebra structure
 on the space of all bounded Borel measures on $X_q$
and an associated commutative hypergroup structure. These hypergroups generalize the known disk
hypergroups for $q=1$ which were 
  studied for instance in \cite{AT}, \cite{Ka}; see also  the monograph
\cite{BH}. 

As ususal, the basis for analytic continuation will be 
 Carlson's theorem,  which we 
recapitulate for the reader's convenience from \cite{Ti}:

\begin{theorem}\label{continuation} Let $f$ be holomorphic in a neighbourhood
of
$\{z\in \mathbb C:{\rm Re\>} z \geq 0\}$ satisfying $f(z) =
O\bigl(e^{c|z|}\bigr)$
 for some $c<\pi$. 
If $f(z)=0$ for all nonnegative integers $z$, then $f$ is identically zero on
$\{{\rm Re\>}  z>0\}$.
\end{theorem}

 It is now straightforward, but a bit nasty in detail to prove
that the
 product formula of 
  Theorem 
 \ref{prop-group-torus-convo-Xq} can be extended analytically with respect to the
variable $p$. For the necessary exponential bounds,
one has to use  that the  coefficients of the Jacobi polynominals $P_\lambda(k;t) = c(\lambda+\rho(k),k)^{-1}R_\lambda(k;t)$ 
 are rational in the multiplicity $k$, see Par. 11 of \cite{M}.
 As the arguments are the same as those in the proof of  Theorem 
4.1 of \cite{R2} and Theorem 3.2 of \cite{V9}, we skip the details.
We obtain:  

\begin{theorem}\label{general-twodim-prod-form}
Let $p\in]2q-1,\infty[$, $\lambda\in P_+$, and $l\in\mathbb Z$.
Then the functions
$$\phi_{\lambda,l}^p([\cos t,z])=z^l\cdot\prod_{j=1}^q \cos^{|l|} t_j\cdot
 R_\lambda(k;t) $$
on $X_q$ with multiplicity $k=(p-q-|l|,\> \frac12 +|l|,\> 1)\in K^{reg}$ satisfy the
product formula
\begin{align}
\phi_{\lambda,l}^p([\cos t,z])\cdot\phi_{\lambda,l}^p([\cos t^\prime, z^\prime]) \,=\, 
\frac{1}{\kappa_p}&\int_{B_q}\int_{SU(q)} 
\phi_{\lambda,l}^p\bigr(\big[\sigma_{sing}(-\sin\underline t \,w\sin\underline t^\prime +
\,\cos \underline t\,v\cos \underline t^\prime),\notag  \\
zz^\prime\cdot \arg\Delta(-\sin\underline t &\,w\sin\underline t^\prime +
\cos \underline t\, v \cos \underline t^\prime)\big]\bigr)\cdot
\Delta(I_q-w^*w)^{p-2q}\, dvdw.\notag
\end{align}
for $(t,z),(t^\prime,z^\prime)\in A_q\times\mathbb T$.
\end{theorem}

The positive product formula in Theorem \ref{general-twodim-prod-form}
for $p\in]2q-1,\infty[$ leads to a continuous series of probability-preserving
commutative convolution algebras on the cone $ X_q$. In
fact, similar to the noncompact case \cite{V9}, we obtain commutative
hypergroups structures on $X_q$ with the $\phi_{\lambda,l}^p$  ($\lambda\in P_+,
\>l\in\mathbb Z$)
as hypergroup characters. To start with, let us 
briefly recapitulate some notions from hypergroup theory from \cite{J},
\cite{BH}.

\begin{definition} A hypergroup is a locally compact Hausdorff space $X$ with a
  weakly continuous, associative and bilinear convolution $*$ on the Banach space
$M_b(X)$
  of all bounded regular Borel measures on $X$ such that the following
  properties hold:
\begin{enumerate}\itemsep=-1pt
\item[\rm{(1)}] For all $x,y\in X$,  $\delta_x*\delta_y$  is a compactly
  supported probability measure on $X$ such that the support
  $\text{supp}(\delta_x*\delta_y)$ depends continuously on $x,y$ with respect
to the so-called
  Michael topology on the space of compact subsets of $X$ (see \cite{J} for
details).  
\item[\rm{(2)}] There exists a neutral element $e\in X$ with $\delta_x*\delta_e=
\delta_e*\delta_x=\delta_x$ for all $x\in X$. 
\item[\rm{(3)}] There exists a continuous involution $x\mapsto\overline x$ on $X$
  such that  $e\in \text{supp} (\delta_x*\delta_y)$ holds
  if and only if $y=\overline x$, and such that $(\delta_x*\delta_y)^-= \delta_{\overline y}*\delta_{\overline x}$, where for  $\mu\in M_b(X)$, 
  the measure $\mu^-$ denotes the pushforward of $\mu$ under the involution.
\end{enumerate}
\end{definition}

Due to weak continuity and bilinearity, the convolution of  bounded
measures on a  hypergroup is  uniquely determined by the convolution of point
measures. 
A  hypergroup is called commutative if so is the convolution. 
We recall  from \cite{J}  that for a Gelfand pair $(G,K)$, the double coset
 space $G//K$ carries the structure of a commutative hypergroup in a natural way.
For a  commutative hypergroup $X$  the dual space is defined by
$$ \widehat X = \{\varphi\in C_b(X): \,\varphi\not\equiv 0, \, \varphi( x* \overline
y ):=
(\delta_x*\delta_{\overline y})(\varphi) = 
\varphi(x)\overline{\varphi(y)} \,\, \forall\,  x,y\in X\} .$$
Each commutative hypergroup $(X,*)$ 
admits a (up to a multiplicative constant unique)
Haar measure $\omega_X$, which is characterized by
 the condition $\omega_X(f)=\omega_X(f_x)$ for
all continuous, compactly supported $f\in C_c(X)$ and all $x\in X$, 
where $f_x$ denotes  the translate
 $f_x(y)=(\delta_y*\delta_x)(f)$.

Now let $p\in ]2q-1,\infty[$. Using the  positive product formula
of Theorem \ref{general-twodim-prod-form}, 
we  introduce the convolution of point measures 
on  $X_q$ by
\begin{align}\label{def-convolution-aq}
(\delta_{[\cos t,z]}*_p \delta_{[\cos t^\prime,z^\prime]})(f):=
 & \,\frac{1}{\kappa_p}\int_{B_q}\int_{SU(q)} 
f\bigl(\big[\sigma_{sing}(-\sin\underline t \,w \sin\underline t^\prime \,+\,
\cos \underline t\, v\,\cos \underline t^\prime
),\notag  \\
zz^\prime\cdot \arg\Delta &(-\sin\underline t \,w\sin\underline t^\prime \,+\,
\cos \underline t\, v \cos \underline t^\prime)\big]\bigr)
\Delta(I_q-w^*w)^{p-2q} dv dw
\end{align}
for $f\in C_b(X_q)$.

\begin{theorem}\label{mainhypergroup}
Let $q\ge 1$ be an integer and $p\in]2q-1,\infty[$. 
Then the convolution $*_p$ defined in (\ref{def-convolution-aq}) 
extends uniquely to a bilinear,
 weakly continuous, commutative  convolution on the Banach space
 $M_b(X_q)$. 
This convolution is also associative, and $(X_q, *_{p})$ is a 
commutative hypergroup with $[(1, \ldots, 1),1]$ as identity and with the involution
 $\overline{[r,z]}:= [r,\overline z]$. A Haar measure of the hypergroup $(X_q, *_{p})$ is given by
  \[ d\omega_p([r,z]) = \prod_{j=1}^q r_j(1-r_j^2)^{p-q}\cdot\prod_{1\leq i < j\leq q} 
 (r_i^2 - r_j^2)^2\,drdz\]
 with the Lebesgue measure $dr$ on $\mathbb R^q$ and the normalized 
 Haar measure $dz$ on $\mathbb T.$ Finally, the dual space is
  given by 
 \[ (X_q, *_{p})^\wedge = \{\varphi_{\lambda, l}^p: \lambda\in P_+,\, l\in \mathbb Z\}.\]

 \end{theorem}
 
 Note that $\omega_p$ is the pushforward measure under the mapping
$(t,z)\mapsto [\cos t, z], \, A_q\times \mathbb T \to X_q\,$ of the measure
\[ d\widetilde\omega_p(t,z) = \prod_{j=1}^q \cos t_j\sin^{2p-2q+1}t_j 
\cdot\prod_{1\leq i < j\leq q} (\cos(2t_i)-\cos(2t_j))^2 dtdz \]
on $A_q\times \mathbb T.$

For the proof of Theorem \ref{mainhypergroup}, consider the measure $\omega_p$ defined above.
We start with the following observation:

\begin{lemma}\label{orthobasis} The functions
\[ \varphi_{\lambda,l}^p \quad (\lambda \in P_+, \, l\in \mathbb Z)\]
form an orthogonal basis of $L^2(X_q, \omega_p),$ and their $\mathbb C$-span is dense in the space $C(X_q)$ 
of continuous functions on $X_q$ with respect to $\|.\|_\infty.$ 
 \end{lemma}

\begin{proof}
 For $k=k(p,q,l)$, the Heckman-Opdam polynomials $R_\lambda(k;t)$ are  orthogonal on
the alcove $A_q$ with respect to the weight
\begin{align*} \delta_k(t) = &\,\prod_{\alpha\in R_+} \big\vert e^{i\langle\alpha,t\rangle/2} - 
e^{-i\langle\alpha,t\rangle/2}\big\vert^{2k_\alpha} \\
 =\,  & const\cdot\prod_{j=1}^q \sin^{2p-2q+1}\!t_j\, \cos^{2|l|+1}\! t_j \cdot\prod_{1\leq i<j\leq q} (\cos(2t_i) - \cos(2t_j))^2
   \end{align*}
This immediately implies
that the functions $(t,z)\mapsto\varphi_{\lambda,l}^p([\cos t,z])$ are orthogonal on 
$A_q\times \mathbb T$
with respect to $\widetilde\omega_p.$ 
Let $V$ denote the subspace of $C(X_q)$ spanned by
$\{\varphi_{\lambda, l}^p: \lambda\in P_+, \, l\in \mathbb Z\}.$ Clearly, $V$ is an algebra which is stable 
under complex conjugation. Further, $V$ separates points on $X_q$, 
because the $R_\lambda(k;\,.\,)$ span the  space of Weyl group invariant trigonometric polynomials.
 Hence by the Stone-Weierstra\ss\ theorem, $V$ is dense in $C(X_q)$ with respect
to $\|.\|_\infty$. The rest is obvious.
\end{proof}

 \begin{proof}[Proof of Theorem \ref{mainhypergroup}]
 The statements are clear for integer values of $p$, where $*_p$ is just
the convolution of the double coset hypergroup $SU(p+q)//SU(p)\times SU(q).$ 
>From the  definition of the convolution for general $p$ we see that 
$\delta_{[r,z]} *_{p} \delta_{[r^\prime, z^\prime]}$ is a probability measure
and that the mapping $([r,z],[r^\prime, z^\prime])\mapsto \delta_{[r,z]} *_{p} \delta_{[r^\prime, z^\prime]}$
is weakly continuous. 
It is now standard (see \cite{J}) to extend  the convolution of point measures 
uniquely in a bilinear, weakly
continuous way to a probability preserving convolution on
$M_b(X_q)$.
For commutativity and associativity, 
it suffices to  consider point measures. Let 
$[r_i, z_i]\in X_q\,,\, 1\leq  i\leq 3.$ Then for
 $f= \varphi_{\lambda,l}^p$ with $\lambda\in P_+$ and $l\in \mathbb Z$ we have
\[ (\delta_{[r_1,z_1]}*_{p} \delta_{[r_2,z_2]})(f) = f([r_1,z_1]) f([r_2,z_2]) = 
(\delta_{[r_2,z_2]}*_{p} \delta_{[r_1,z_1]})(f)\]
and in the same way,
\[((\delta_{[r_1,z_1]}*_{p} \delta_{[r_2,z_2]})*_{p}
 \delta_{[r_3,z_3]})(f) =
(\delta_{[r_1,z_1]}*_{p} (\delta_{[r_2,z_2]}*_{p} \delta_{[r_3,z_3]}))(f).\]
By Lemma \ref{orthobasis}, both identities remain valid for arbitrary $f\in C(X_q).$
The remaining hypergroup axioms are immediate from the fact that 
the supports of convolution products  of point measures are independent of $p$.
For the statement on the Haar measure, the argumentation is similar to \cite{RR}. Notice first that by definitition of hypergroup translates, 
the identity 
\begin{equation*} \int_{X_q} f_x\, d\omega_p = \int_{X_q} f d\omega_p \quad \text{for all } 
x\in X_q.\end{equation*}
holds for $f= \varphi_{\lambda, l}^ p$ with arbitrary 
$ \lambda\in P_+, \,  l\in \mathbb Z.$ In view of Lemma \ref{orthobasis}, it extends to arbitrary $f\in C(X_q)$,  
hence $\omega_p$ is a Haar measure. Finally, it is clear that the $\varphi_{\lambda, l}^p$ are hypergroup characters.
There are no further ones, because the characters of a compact hypergroup are orthogonal with 
respect to its Haar measure.

 \end{proof}

The hypergroups $(X_q\,, *_{p})$ have a prominent subgroup.
For this we recall that a closed, non-empty subset
 $H\subset X_q$ is a subgroup if  for all $x,y\in H$,
 $\delta_x*_p\delta_{\overline y}$ is a point measure with support in $H$.
It is clear from (\ref{def-convolution-aq}) that 
\[H:=\{[1,z] =(z,1,\ldots, 1): z\in\mathbb T\}\] 
is a subgroup 
of  $(X_q, *_{p})$ which is isomorphic to the torus group
$\mathbb T$. 

The cosets
 \[x*_pH:=\bigcup_{y\in H}{\rm supp}\> (\delta_x*_p\delta_{ y}), 
\quad x\in X_q \]
form a disjoint decomposition of $X_q$, and the quotient 
\[X_q/H:=\{x*_p\!H: \> x\in X_q \}\]
 is again a locally compact Hausdorff space with respect to the quotient
topology, c.f. Section 10.3. of \cite{J}. Moreover,
\begin{equation}\label{quotient-allg}
(\delta_{x*_pH}*_p\delta_{y*_pH})(f):=\int_X f(z*_pH)\>
d(\delta_{x}*_p\delta_{y})(z), 
\quad x,y\in X_q\, , \, f\in C_b(X_q/H)
\end{equation}
establishes a well-defined quotient convolution and an associated commutative
quotient
hypergroup; see
 \cite{J}, \cite{V1}, and the references given there. We may identify $X_q/H$ 
 topologically with the alcove 
$A_q$ via $[\cos t,z]*_p\!H \mapsto t.$ It is then immediate from
 \eqref{def-convolution-aq} that  the quotient convolution on $A_q$ derived from $\ast_p$ is
 given by
\begin{align}\label{def-convolution-aqres}
(\delta_{t}&*_p \delta_{t^\prime})(f)=\\
&=\frac{1}{\kappa_p}\int_{B_q}\int_{SU(q)} 
f\bigl(\arccos(\sigma_{sing}(-\sin\underline t \,w\sin\underline t^\prime \,+\,
\cos \underline t\, v \cos \underline t^\prime
))\bigr) \cdot\Delta(I_q-w^*w)^{p-2q} dv dw\notag
\end{align}
for 
$t, t^\prime\in A_q$ and $f\in C_b(A_q).$
These are precisely the hypergroup convolutions studied in Section 6 of \cite{RR}.
 For integers $p\ge 2q$, this connection just reflects the fact that for 
 $G=SU(p+q), K= S(U(p)\times U(q))$ and $K_1= SU(p)\times SU(q),$ we have
 $$ (G// K_1)/(K//K_1)\simeq G//K$$
 as hypergroups. This a fact which holds for general commutative double coset hypergroups,  see Theorem 14.3 of \cite{J}.
 
 \begin{remarks}
1.  We mention at this point that 
  the Haar measure of the hypergroup $(X_q, *_p)$, which was obtained in Theorem \ref{mainhypergroup} by 
  an orthogonality argument, can alternatively be calculated by using 
  Weil's integration  formula for Haar measures on hypergroups and their quotients 
(see \cite{Her}, \cite{V1}). In the same way as in the non-compact case treated in \cite{V9}, the Haar measure can thus be
obtained from the known Haar measure of the quotient $(X_q, *_p)/H$.
For integers $p\ge 2q$, the hypergroup $(X_q, \ast_p)$ can be identified with 
the double coset hypergroup 
$SU(p+q)//SU(p)\times SU(q),$ and its 
Haar measure therefore coincides by construction (see \cite{J}) with  the 
pushforward measure 
of the Haar measure on $SU(p+q)$ under the canonical projection
\[SU(p+q)\to SU(p+q)//SU(p)\times SU(q)\simeq X_q\,.\]

\smallskip\noindent
2. Concerning their structure, the hypergroups $(X_q, \ast_p)$ are also closely related to 
continuous family 
of commutative hypergroups
$(C_q\times \mathbb R, \ast_p)$ with $ p\geq  2q-1$ and  the $BC_q$-Weyl chamber
$\, C_q = \{(x_1, \ldots, x_q)\in \mathbb R^q: x_1 \geq \ldots \geq x_q\geq 0\}$
which were studied in \cite{V2}. For integral $p$, $(C_q\times \mathbb R, \ast_p)$ 
is an orbit hypergroup
under the action of $U(p)\times U(q)$ on the Heisenberg group
$M_{p,q}(\mathbb C)\times \mathbb R$. The characters are given in terms of 
multivariable Bessel- and Laguerre functions.
\end{remarks}

\section{Continuous product formulas for Heckman-Opdam Jacobi polynomials}

Fix the rank $q\ge1$ and a parameter  $p\in ]2q-1,\infty[$. Recall 
 that for  $l\in\mathbb Z$ the functions
\[\phi_{\lambda,l}^p([\cos t, z]) = z^l\cdot
\prod_{j=1}^q \cos^{|l|}\! t_j\cdot R_\lambda(k(p,q,l);t)\]
satisfy the product formula of Theorem \ref{general-twodim-prod-form}.
We shall now extend this fomula to exponents $l\in\mathbb R$ via Carlson's theorem, and
write it down as a product formula for the Jacobi polynomials 
$R_\lambda(k;t)$ with $k=k(p,q,l)$. This will work out smoothly only for
non-degenerate arguments $\,t, t^\prime\in A_q$ with $t_1, t_1^\prime \ne \pi/2$. Notice first
 that for a product formula for the Jacobi polynomials, we may restrict
our attention to $l\in[0,\infty[$ as $k(p,q,l)$ depends on $|l|$ only.
In the following, we shall use the abbreviation
\[d(t,t^\prime;v,w):= -\sin\underline t \,w \sin\underline t^\prime \,+\,\cos \underline t\, v \cos \underline t^\prime\,.\]
The main result of this section is

\begin{theorem}
Let  $q\ge1$ be an integer,  $p\in ]2q-1,\infty[$,  $l\in[0,\infty[$,
 and     $k=k(p,q,l)$. Then for all $\lambda\in P_+$ and 
 $\, t, t^\prime\in A_q$ with $t_1, t_1^\prime\ne \pi/2$, 
\begin{align}\label{prod-formel-allg-Jacobi}
R_\lambda(k; t) R_\lambda(k; t^\prime)\,
=\,\frac{1}{\kappa_p}\int_{B_q}\int_{SU(q)} &
R_\lambda\bigl(k;\arccos(\sigma_{sing}(d(t,t^\prime;v,w)))\bigr) \cdot\\
 &\cdot {\rm Re}\biggl[\biggl(\frac{\Delta(d(t,t^\prime;v,w))}{\Delta(\cos \underline
t)\cdot
\Delta(\cos \underline t^\prime)}\biggr)^{\!l}\,\biggr]
\cdot\Delta(I_q-w^*w)^{p-2q} dvdw.\notag
\end{align}
\end{theorem}

\begin{proof} For  $l\in \mathbb Z_+$, $z=z^\prime =1$ and  $t, t^\prime\in A_q$, the product
  formula in Theorem  \ref{general-twodim-prod-form} implies
\begin{align}\label{hilf-prod-f}
&\Bigl(\prod_{j=1}^q \cos t_j  \cos t_j^\prime\Bigr)^l \cdot 
 R_\lambda(k; t) R_\lambda(k; t^\prime)=\\
&=\frac{1}{\kappa_p}\int_{B_q}\int_{SU(q)}\Delta(d(t,t^\prime;v,w))^l \cdot 
 R_\lambda\bigl(k;\arccos(\sigma_{sing}(d(t,t^\prime;v,w)))\bigr)
\cdot\Delta(I_q-w^*w)^{p-2q} dv dw.
\notag
\end{align}
Our condition on $t,t^\prime$ assures that both sides of (\ref{hilf-prod-f}) are analytic in $l$ for $ {\rm
  Re} (l)> 0$.
 Moreover, by Section 11 of \cite{M}, the coefficients of the Jacobi polynomials 
 $R_\lambda(k;t)$ with respect to the exponentials  $e^{i\langle\mu,t\rangle}\,(\mu\in P_+)$
are 
rational in $k$. Therefore
both
sides of  (\ref{hilf-prod-f}) satisfy the growth condition of Carlson's
theorem. This implies that (\ref{hilf-prod-f}) remains correct for all
$l\in[0,\infty[$. As
$R_\lambda(k;t)$ is real for all $\lambda\in P_+$ and $t\in A_q$, the 
 claimed  product formula now follows from  (\ref{hilf-prod-f}) by taking real parts.
 
\end{proof}

Contrary to  the non-compact case in Section 5 of \cite{V9}, it seems to be
difficult in the present setting to derive positivity of the product formula \eqref{prod-formel-allg-Jacobi}
except for the known
case $l=0$.
This problem already appears in rank one, i.e. for $q=1$. We discuss this case  for illustration.

\begin{example}\label{1-dim-bsp-fortsetzung}
 Let $q=1$, $p\ge 2q-1= 1$, $l\in[0,\infty[$, and  $k=k(p,q,l)= (p-1-l, \frac{1}{2}+l).$
Resuming  the notions from Example \ref{q1--example-1}, we have
\[ \alpha=k_1+k_2-1/2=p-1>0,
\quad
 \beta=k_2-1/2=l.\]
Consider the classical (normalized) one-dimensional Jacobi polynomials
$R_n^{(\alpha,\beta)}$ with
\[R_n^{(\alpha,\beta)}(\cos 2\theta)=R_{2n}(k;\theta),\]
c.f. \eqref{ident-jacobi}.
With the identity $\cos 2\theta=2\cos^2\theta -1$, product formula
(\ref{prod-formel-allg-Jacobi})  becomes
\begin{align}\label{Jacobi_1} 
R_n^{(\alpha,\beta)}&(\cos 2\theta) R_n^{(\alpha,\beta)}(\cos 2\theta^\prime)
=\\
=&\frac{\alpha}{\pi} \int_0^1 \int_{-\pi}^{\pi} 
R_n^{(\alpha, \beta)} (2|b+are^{i\varphi}|^2 -1) \cdot 
 \frac{(b+are^{i\varphi})^\beta}{b^\beta} \cdot r(1-r^2)^{\alpha-1} dr d\varphi
\notag\end{align}
for $\theta,\theta^\prime\in [0,\pi/2[$, with   $\,a:= \sin \theta\sin\theta^\prime$
and $b: = \cos\theta\cos\theta^\prime>0$. 
Following Section 5 of \cite{K3}, we use the change of variables $(r,\phi)\mapsto (t,\psi) $ with
 \[ b+are^{i\varphi} = te^{i\psi} \text{ and }\, t\geq 0.\]
 Then  $a^2r\> dr\> d\phi= t\>
dt\> d\psi$ and 
identity \eqref{Jacobi_1} becomes,
 for $0< \theta,\theta^\prime < \pi/2$:
\begin{align}\label{Jacobi_2} 
R_n^{(\alpha,\beta)}&(\cos 2\theta) R_n^{(\alpha,\beta)}(\cos 2\theta^\prime)
=\\
=&\frac{\alpha}{\pi}\cdot\frac{1}{b^\beta a^{2\alpha}}
\int_0^\infty\int_{-\pi}^{\pi} R_n^{(\alpha,\beta)}(2t^2-1)
(te^{i\psi})^\beta
(a^2-b^2-t^2+2bt\cos\psi)_+^{\alpha-1} t\> dt\> d\psi\notag\\ 
=&\frac{\alpha}{\pi}\cdot\frac{2}{b^\beta a^{2\alpha}}\int_0^1
R_n^{(\alpha,\beta)}(2t^2-1) \Bigl(\int_{0}^{\pi} 
e^{i\beta\psi}(a^2-b^2-t^2+2bt\cos\psi)_+^{\alpha-1}d\psi\Bigr)t^{\beta+1} dt.
\notag\end{align}
Here the notation 
\[ (x)_+^\lambda = \begin{cases} x^\lambda & \text{ if $x>0$,}\\
                    0 & \text{ if $x\leq 0$}
                   \end{cases}\]
is used.
Notice for the last equality that $\,t=|b+are^{i\phi}|\le a+b\le 1$. Notice also
that 
for $\theta=0$ or  $\theta^\prime=0$, the product formula
degenerates in a trivial way due to $R_n^{(\alpha,\beta)}(1)=1$.

On the other hand, for $\alpha>\beta$, the Jacobi polynomials satisfy 
the well-known positive product formula (\cite{K1})
\begin{align}\label{Jacobi_3} 
R_n^{(\alpha,\beta)}&(\cos 2\theta)R_n^{(\alpha,\beta)}(\cos 2\theta^\prime) =\\
=&c_{\alpha,\beta}\int_0^1\int_0^\pi
R_n^{(\alpha,\beta)}(2|b+are^{i\varphi}|^2-1)\cdot 
(1-r^2)^{\alpha-\beta-1}r^{2\beta+1}\sin^{2\beta}\varphi\, dr d\varphi
%\notag 
%\\ 
%=&\frac{c_{\alpha,\beta}}{2}\int_0^1\int_{-\pi}^\pi
% R_n^{(\alpha,\beta)}(2|b+are^{i\varphi}|^2-1)\cdot 
%(1-r^2)^{\alpha-\beta-1}r^{2\beta+1}\sin^{2\beta}\varphi\>  dr \> d\varphi
\notag\end{align}
with some  constant $c_{\alpha,\beta}>0$. 
By the same substitution as above this can be brought into kernel form,
\begin{align}\label{Jacobi_4} 
R_n^{(\alpha,\beta)}&(\cos 2\theta)R_n^{(\alpha,\beta)}(\cos 2\theta^\prime) = \\
&=
\frac{c_{\alpha,\beta}}{a^{2\alpha}}\cdot 
\int_0^1 R_n^{(\alpha,\beta)}(2t^2-1) \Bigl(\int_{0}^\pi 
 (a^2-b^2-t^2+2bt\cos\psi)_+^{\alpha-\beta-1}\sin^{2\beta}\!\psi d\psi\Bigr)
t^{2\beta+1}dt.
\notag\end{align}
As the integrals in \eqref{Jacobi_2} and \eqref{Jacobi_4} are identical for all
$n$, we conclude that, for all $t\in [0,1]$,
\begin{equation} c_{\alpha,\beta}(tb)^\beta \!
\int_{0}^\pi(a^2-b^2-t^2+2bt\cos\psi)_+^{\alpha-\beta-1}\sin^{2\beta}\!\psi
d\psi \,=\, \frac{2\alpha}{\pi} \int_{0}^\pi e^{i\beta\psi} 
 (a^2-b^2-t^2+2bt\cos\psi)_+^{\alpha-1}d\psi.
\end{equation}
 This identity seems not obvious, and  it would be desirable 
 to have an elementary proof for it which  possibly might
 be extended to the higher rank case.
  \end{example} 

So far, a positive product formula such as formula
(\ref{Jacobi_3}) of Koornwinder  seems to be a difficult task in rank $q\geq 2$.  However, at
least the first step above in the case $q=1$, that is from \eqref{Jacobi_1} to
 \eqref{Jacobi_2}, can be partially extended to $q\ge 2$.
Indeed, consider the product formula (\ref{prod-formel-allg-Jacobi})
 for $t,t^\prime\in A_q$. We define the matrices 
 \[a_1:=\sin \underline t\,, \,   
a_2:=\sin \underline t^\prime\,,\,  b= b(v):= \cos \underline t\, v \cos \underline t^\prime\,  \in M_q(\mathbb C)\]
 and consider
 the polar decomposition $\,b-a_1wa_2=: ru\,$ with positive semidefinite
  $r\in M_q(\mathbb C)$ and $u\in U(q)$. We now carry out the change of variables 
 in two steps.
First, for $\,\widetilde w:= a_1wa_2\in M_q(\mathbb C)$, we have
$\,d\widetilde w= const\cdot\Delta(a_1a_2)^{2q}\> dw$. Moreover, the polar decomposition
$\,b-\widetilde w=\sqrt{r}u\,$ (for non-singular $r$) leads to 
$\,d\widetilde w= const\cdot drdu,$  where  $dr$ means integration with respect to the
Lebesgue measure on the cone $\Omega_q$ of positive definite matrices 
as an open subset of the vector space of all Hermitian matrices, and $du$ is the
normalized Haar measure of $U(q)$; see Proposition XVI.2.1 of \cite{FK}. Formula
 (\ref{prod-formel-allg-Jacobi}) now reads 
\begin{align}\label{prod-formel-allg-Jacobi-kern}
R_\lambda(k; t)R_\lambda(k; t^\prime)= & \notag\\
=\,const\cdot\int_{\Omega_q} &
R_\lambda\bigl(k;\arccos(\sigma_{sing}(\sqrt{r}))\bigr) \cdot
\frac{\Delta(r)^{l/2}}{\Delta(\cos \underline t)^l\Delta(\cos \underline t^\prime)^l
\Delta(\sin \underline t)^{2q}\Delta(\sin \underline t^\prime)^{2q}}\cdot
\notag\\
&\cdot\Bigl(\int_{SU(q)}\int_{U(q)} 
\Delta\bigl(H(t,t^\prime,r,u,v)_+\bigr)
\Delta(u)^l\,dv du\Bigr) dr
\end{align}
where 
$$H(t,t^\prime,r,u,v)=I_q \,-\, a_2^{-1}\bigl(b(v)^*-u^*\sqrt{r}\bigr) a_1^{-2}
\bigl(b(v)-\sqrt{r}u\bigr) a_2^{-1}$$
and the subscript $+$ means that this term is put zero for matrices which
 are not positive definite. The analysis of the origin of these
 formulas shows that in the outer integral, $r$ in fact runs through the
 set  $\{r\in\Omega_q:\> I_q-r >0\}$. 
One may speculate that (\ref{prod-formel-allg-Jacobi-kern})
 can be brought  into a kernel form 
by replacing the
integration over $r\in \Omega_q$ by an integration over the spectrum of $r$ as a
subset of $\{\rho\in\mathbb R^q:\> 1 \ge\rho_1\ge\cdots\ge\rho_q\ge 0\}$.

\end{document}